\documentclass[10pt,leqno]{article}%
\usepackage{graphicx}
\usepackage{amsmath}
\usepackage{amsfonts}
\usepackage{amssymb}%
\newtheorem{theorem}{Theorem}[section]

\newtheorem{remark}{Remark}[section]
\newtheorem{example}{Example}[section]

\newtheorem{lemma}[theorem]{Lemma}

\newenvironment{proof}[1][Proof]{\textbf{#1.} }{\ \rule{0.5em}{0.5em}}

\begin{document}

\title{Systems described by Volterra type integral operators}
\author{Dorota Bors, Andrzej Skowron and Stanis\l aw Walczak\\{\small Faculty of Mathematics and Computer Science} \\{\small University of Lodz}\\{\small Banacha 22, 90-238 Lodz, Poland}\\{\small email: \{{bors, skowroa, stawal\}}@math.uni.lodz.pl}}
\date{}
\maketitle

\begin{abstract}
In the paper some sufficient condition for the nonlinear integral operator of
the Volterra type to be a diffeomorphism defined on the space of absolutely
continuous functions are formulated.\ The proof relies on consideration of the
linearized equation together with Palais-Smale condition, thus a combination
of topological and variational methods is used. The applications of the result
to the control systems with feedback and to the specific nonlinear Volterra
equation are presented.


\end{abstract}

\section{Introduction}

Let $X$ be a Banach space and let $V:X\rightarrow X$ be some operator. An
operator $V$ is said to be a diffeomorphism, if $V(X)=X,$ there exists an
inverse operator $V^{-1}:X\rightarrow X,$ and the operators $V$ and $V^{-1}$
are Fr\'{e}chet differentiable at any point $x\in X.$

\noindent In the paper we shall consider the nonlinear integral operators of
the form
\begin{equation}
V\left(  x\right)  \left(  t\right)  =x\left(  t\right)  +\int\limits_{\alpha
}^{t}v\left(  t,\tau,x\left(  \tau\right)  \right)  d\tau\label{0}%
\end{equation}
where $t\in\left[  \alpha,\beta\right]  ,$ $v:P_{\Delta}\times\mathbb{R}%
^{n}\rightarrow\mathbb{R}^{n}$, $n\geq1$ and $P_{\Delta}=\left\{  \left(
t,\tau\right)  \in\left[  \alpha,\beta\right]  \times\left[  \alpha
,\beta\right]  :\tau\leq t\right\}  $ and $x\in AC_{0}^{2}.$ By $AC_{0}%
^{2}=AC_{0}^{2}\left(  \left[  \alpha,\beta\right]  ,\mathbb{R}^{n}\right)  $
we shall denote the space of absolutely continuous functions defined on the
interval $\left[  \alpha,\beta\right]  $ satisfying conditions $x\left(
\alpha\right)  =0$ and $x^{\prime}\in L^{2}=L^{2}\left(  \left[  \alpha
,\beta\right]  ,\mathbb{R}^{n}\right)  $ with the norm given by the formula%
\[
\left\Vert x\right\Vert _{AC_{0}^{2}}^{2}=\int\limits_{\alpha}^{\beta
}\left\vert x^{\prime}\left(  t\right)  \right\vert ^{2}dt.
\]
Under some appropriate assumptions, to be specified in the next sections,
imposed on the function $v$, it is feasible to formulate some sufficient
condition for the operator $V:AC_{0}^{2}\rightarrow AC_{0}^{2}$ to be\ a
diffeomorphism. One can rephrase, underlying the property of the continuous
dependence on parameters, the definition of being a diffeomorphism as the
operator $V:AC_{0}^{2}\rightarrow AC_{0}^{2}$ is a diffeomorphism if and only
if\newline(a) the operator $V$ is Fr\'{e}chet differentiable at every point
$x\in AC_{0}^{2}$ and for every $y\in AC_{0}^{2}$ there exists a unique
solution $x=x_{y}\in AC_{0}^{2}$ to the equation $V\left(  x\right)  =y$ and
the solution $x_{y}$ depends continuously on parameter $y\in AC_{0}^{2}%
,$\newline(b) the operator $AC_{0}^{2}\backepsilon y\rightarrow x_{y}\in
AC_{0}^{2}$ is Fr\'{e}chet differentiable.

\noindent The systems satisfying condition (a) are often referred to as stable
ones and well-posed. If, additionally, the condition (b) is guaranteed then
the system is called a robust one, cf. \cite{SanSzn}.

\noindent The theory of integral operators and integral equations forms an
extensive area of mathematics. The applications to physics, biology and
technical sciences are vast. In particular, Volterra equations find
applications in mechanics, demography and epidemiology, to mention only,
specific areas of: population dynamics, the spread of an epidemics, the study
of viscoelastic materials with memory, the cosmic ray transport of charged
particles in a turbulent plasma, the nuclear reactor dynamics with feedback or
the general control theory for the systems with feedback loops, cf.
\cite{AtaPil, Chr, Cor,CraNoh, GriLonSta, Lon1, Lon2, Pod1, Pod2, Pru,
RenHruNoh} and references therein.

The structure of the paper reads as follows. In current section the set-up of
notation and terminology is presented. Section 2 contains some regularity
assumptions imposed on the function $v$ and some auxiliary lemmas. In Section
3, we focus our attention on stating a sufficient condition for the functional
defined by the Volterra type operator to satisfy Palais-Smale condition. Some
sufficient condition for the nonlinear integral operator of the Volterra type
to be a diffeomorphism can be found in Section 4. Finally, the examples of
integral operators of the Volterra type satisfying some sufficient condition
for being a diffeomorphism are presented.

\section{Preliminaries}

In this section we impose fundamental assumptions on the function $v$ defining
the operator $V$. In what follows, we will need the following conditions:

\begin{enumerate}
\item[(A1)] (a) the function $v\left(  \cdot,\tau,\cdot\right)  $ is
continuous on the set $G:=[\alpha,\beta]\times\mathbb{R}^{n}$ for a.e.
$\tau\in\lbrack\alpha,\beta],$\newline\ \ \ \ \ \ (b) there exists
$v_{t}\left(  \cdot,\tau,\cdot\right)  $ and it is continuous on $G$ for a.e.
$\tau\in\lbrack\alpha,\beta],$\newline\ \ \ \ \ \ (c) there exists
$v_{x}\left(  \cdot,\tau,\cdot\right)  $ and it is continuous on set $G$ for
a.e. $\tau\in\lbrack\alpha,\beta],$\newline\ \ \ \ \ \ (d) there exists
$v_{tx}\left(  \cdot,\tau,\cdot\right)  $ and it is continuous on set $G$ for
a.e. $\tau\in\lbrack\alpha,\beta];$

\item[(A2)] (a) the function $v\left(  t,\tau,x\right)  $ is measurable with
respect to $\tau$ and locally bounded with respect to $x,$ i.e. for every
$\rho>0$ there exist $l_{\rho}>0$ such that for $\left(  t,\tau\right)  \in
P_{\Delta}$ and $x\in B_{\rho}=\left\{  x\in\mathbb{R}^{n}:\left\vert
x\right\vert \leq\rho\right\}  $ we have $\left\vert v\left(  t,\tau,x\right)
\right\vert \leq l_{\rho},$\newline\ \ \ \ \ \ (b) the function $v_{t}\left(
t,\tau,x\right)  $ is measurable with respect to $\tau$ and locally bounded
with respect to $x,$\newline\ \ \ \ \ \ (c) the function $v_{x}\left(
t,\tau,x\right)  $ is measurable with respect to $\tau$ and locally bounded
with respect to $x,$\newline\ \ \ \ \ \ (d) the function $v_{tx}\left(
t,\tau,x\right)  $ is measurable with respect to $\tau$ and locally bounded
with respect to $x.$
\end{enumerate}

Before we formulate the main result of the paper we proceed with some
auxiliary lemmas.

\begin{lemma}
\label{lemat2.1}If the function $v$ satisfies $\left(  A1a\right)
,(A1b),(A2a)$ and $(A2b)$, then the operator $V:AC_{0}^{2}\rightarrow
AC_{0}^{2}$ is well-defined by $\left(  \ref{0}\right)  .$
\end{lemma}

\begin{proof}
Let $x_{0}\in AC_{0}^{2}$. It is enough to demonstrate that the function $u$
defined by formula%
\[
u\left(  t\right)  =\int_{\alpha}^{t}v\left(  t,\tau,x_{0}\left(  \tau\right)
\right)  d\tau
\]
is absolutely continuous and the derivative of $u$ is square-integrable$.$
This can be seen by observing that%
\begin{align*}
&  \sum\limits_{i=1}^{N}\left\vert u\left(  t_{i+1}\right)  -u\left(
t_{i}\right)  \right\vert =\sum\limits_{i=1}^{N}\left\vert \int_{\alpha
}^{t_{i}}\left[  v\left(  t_{i+1},\tau,x_{0}\left(  \tau\right)  \right)
-v\left(  t_{i},\tau,x_{0}\left(  \tau\right)  \right)  \right]  d\tau
+\int_{t_{i}}^{t_{i+1}}v\left(  t_{i+1},\tau,x_{0}\left(  \tau\right)
\right)  d\tau\right\vert \\
&  \leq\int_{\alpha}^{\beta}l_{\rho}d\tau\sum\limits_{i=1}^{N}\left\vert
t_{i+1}-t_{i}\right\vert +\sum\limits_{i=1}^{N}\int_{t_{i}}^{t_{i+1}}l_{\rho
}d\tau\leq l_{\rho}\left(  \beta-\alpha+1\right)  \sum\limits_{i=1}%
^{N}\left\vert t_{i+1}-t_{i}\right\vert
\end{align*}
where $\alpha\leq t_{1}<t_{2}<...<t_{i}<t_{i+1}<...<t_{N}<t_{N+1}\leq\beta.$
Consequently, for any $x_{0}\in AC_{0}^{2}$ the function $u$ is absolutely
continuous and therefore, for almost any $t\in\left(  \alpha,\beta\right)  ,$
there exists the derivative of $u$ and
\[
\int\limits_{\alpha}^{\beta}\left\vert u^{\prime}\left(  t\right)  \right\vert
^{2}dt\leq2\int\limits_{\alpha}^{\beta}\left\vert v\left(  t,t,x_{0}\left(
t\right)  \right)  \right\vert ^{2}dt+2\int\limits_{\alpha}^{\beta}\left(
\int\limits_{\alpha}^{t}\left\vert v_{t}\left(  t,\tau,x_{0}\left(
\tau\right)  \right)  \right\vert d\tau\right)  ^{2}dt.
\]
By assumptions $\left(  A2a\right)  $ and $\left(  A2b\right)  ,$ the
derivative of $u$ belongs to $L^{2}.$
\end{proof}

\begin{lemma}
For any $x\in AC_{0}^{2}$ we have%
\begin{align}
\left\vert x\left(  t\right)  \right\vert  &  \leq\left(  t-\alpha\right)
^{\frac{1}{2}}\left\Vert x\right\Vert _{AC_{0}^{2}}\text{ for }t\in\left[
\alpha,\beta\right]  ,\text{ }\label{((2.1))}\\
\int\limits_{\alpha}^{\beta}\left\vert x\left(  t\right)  \right\vert ^{2}dt
&  \leq\frac{1}{2}\left(  \beta-\alpha\right)  ^{2}\left\Vert x\right\Vert
_{AC_{0}^{2}}^{2}.\nonumber
\end{align}

\end{lemma}

\begin{proof}
Immediately from the Schwarz inequality, for any $t\in\left[  \alpha
,\beta\right]  ,$ we get%
\[
\left\vert x\left(  t\right)  \right\vert \leq\int\limits_{\alpha}%
^{t}\left\vert x^{\prime}\left(  \tau\right)  \right\vert d\tau\leq\left(
t-\alpha\right)  ^{\frac{1}{2}}\left\Vert x\right\Vert _{AC_{0}^{2}}%
\]
and subsequently
\[
\int\limits_{\alpha}^{\beta}\left\vert x\left(  t\right)  \right\vert
^{2}dt\leq\left\Vert x\right\Vert _{AC_{0}^{2}}^{2}\int\limits_{\alpha}%
^{\beta}\left(  t-\alpha\right)  dt=\frac{1}{2}\left(  \beta-\alpha\right)
^{2}\left\Vert x\right\Vert _{AC_{0}^{2}}^{2},
\]
and this is precisely the assertion of the lemma.
\end{proof}

For the operator $V:AC_{0}^{2}\rightarrow AC_{0}^{2}$ defined by $\left(
\ref{0}\right)  $ we have the following differentiability result:

\begin{lemma}
\label{lemat2.2}The operator $V$\ defined by $\left(  \ref{0}\right)  $ is
Fr\'{e}chet differentiable at every point $x_{0}\in AC_{0}^{2}$ and for any
$t\in\left[  \alpha,\beta\right]  $ we have%
\begin{equation}
V^{\prime}\left(  x_{0}\right)  h\left(  t\right)  =h\left(  t\right)
+\int\limits_{\alpha}^{t}v_{x}\left(  t,\tau,x_{0}\left(  \tau\right)
\right)  h\left(  \tau\right)  d\tau\label{(2.2)}%
\end{equation}
provided that the function $v$ satisfies $\left(  A1a\right)  ,$ $\left(
A1c\right)  ,$ $\left(  A2a\right)  $ and $\left(  A2c\right)  .$
\end{lemma}

\begin{proof}
In view of the definition of $V$, it suffices to show that the operator
\[
V^{0}\left(  x\right)  \left(  t\right)  =\int\limits_{\alpha}^{t}v\left(
t,\tau,x\left(  \tau\right)  \right)  d\tau
\]
is Fr\'{e}chet differentiable. For any $t\in\left[  \alpha,\beta\right]  ,$
any $h\in AC_{0}^{2}$ and some $\theta\in\left[  0,1\right]  $, the
application of the Mean Value Theorem (cf. \cite{IofTih}) enables us to write%
\begin{align*}
V^{0}\left(  x_{0}+h\right)  \left(  t\right)  -V^{0}\left(  x_{0}\right)
\left(  t\right)   &  =\int\limits_{\alpha}^{t}\left[  v\left(  t,\tau
,x_{0}\left(  \tau\right)  +h\left(  \tau\right)  \right)  -v\left(
t,\tau,x_{0}\left(  \tau\right)  \right)  \right]  d\tau\\
&  =\int\limits_{\alpha}^{t}v_{x}\left(  t,\tau,x_{0}\left(  \tau\right)
\right)  h\left(  \tau\right)  d\tau\\
&  +\int\limits_{\alpha}^{t}\left[  \int\limits_{0}^{1}v_{x}\left(
t,\tau,x_{0}\left(  \tau\right)  +\theta h\left(  \tau\right)  \right)
d\theta-v_{x}\left(  t,\tau,x_{0}\left(  \tau\right)  \right)  \right]
h\left(  \tau\right)  d\tau.
\end{align*}
We see from $\left(  \ref{((2.1))}\right)  $ that%
\begin{align*}
&  \left\vert \int\limits_{\alpha}^{t}\left[  \int\limits_{0}^{1}v_{x}\left(
t,\tau,x_{0}\left(  \tau\right)  +\theta h\left(  \tau\right)  \right)
d\theta-v_{x}\left(  t,\tau,x_{0}\left(  \tau\right)  \right)  \right]
h\left(  \tau\right)  d\tau\right\vert \\
&  \leq\int\limits_{\alpha}^{\beta}\int\limits_{0}^{1}\left\vert v_{x}\left(
t,\tau,x_{0}\left(  \tau\right)  +\theta h\left(  \tau\right)  \right)
-v_{x}\left(  t,\tau,x_{0}\left(  \tau\right)  \right)  \right\vert d\theta
d\tau\left(  \beta-\alpha\right)  ^{\frac{1}{2}}\left\Vert h\right\Vert
_{AC_{0}^{2}}.
\end{align*}
It should be noted that the strong convergence in $AC_{0}^{2}$ implies the
uniform convergence in $C.$ Then, by the assumptions of the lemma, it is
enough to apply the Lebesgue Theorem to obtain, if $\left\Vert h\right\Vert
_{AC_{0}^{2}}\rightarrow0,$ the following convergence
\[
\int\limits_{\alpha}^{\beta}\left\vert v_{x}\left(  t,\tau,x_{0}\left(
\tau\right)  +\theta h\left(  \tau\right)  \right)  -v_{x}\left(  t,\tau
,x_{0}\left(  \tau\right)  \right)  \right\vert d\tau\rightarrow0
\]
and therefore
\[
V^{0}\left(  x_{0}+h\right)  \left(  t\right)  -V^{0}\left(  x_{0}\right)
\left(  t\right)  =\int\limits_{\alpha}^{t}v_{x}\left(  t,\tau,x_{0}\left(
\tau\right)  \right)  h\left(  \tau\right)  d\tau+o\left(  h\right)
\]
where $o\left(  h\right)  /\left\Vert h\right\Vert _{AC_{0}^{2}}\rightarrow0$
as $\left\Vert h\right\Vert _{AC_{0}^{2}}\rightarrow0.$ This finishes the proof.
\end{proof}

The rest of this section is devoted to a close study on the linear operator
derived from Fr\'{e}chet differential of the operator $V$ defined by $\left(
\ref{0}\right)  .$ Actually, let $x_{0}\in AC_{0}^{2}$ be a fixed but an
arbitrary function and $T:AC_{0}^{2}\rightarrow AC_{0}^{2}$ be a linear
operator defined, for any $g\in AC_{0}^{2}$ and any $t\in\left[  \alpha
,\beta\right]  ,$ by%
\begin{equation}
\left(  Tg\right)  \left(  t\right)  =\int\limits_{\alpha}^{t}v_{x}\left(
t,\tau,x_{0}\left(  \tau\right)  \right)  g\left(  \tau\right)  d\tau\text{.}
\label{(2.3)}%
\end{equation}
Due to the restrictions imposed on $v$ it is plausible to consider, for any
$k\in\mathbb{N}_{0}$, $t\in\left[  \alpha,\beta\right]  $ and $g\in AC_{0}%
^{2}$, the following sequence of iterations%
\begin{equation}
\left(  T^{k+1}g\right)  \left(  t\right)  =T\left(  T^{k}g\right)  \left(
t\right)  =\int\limits_{\alpha}^{t}v_{x}\left(  t,\tau,x_{0}\left(
\tau\right)  \right)  \left(  T^{k}g\right)  \left(  \tau\right)
d\tau\label{(2.4)}%
\end{equation}
with the first term defined as%
\begin{equation}
\left(  T^{0}g\right)  \left(  t\right)  =g\left(  t\right)  . \label{(2.5)}%
\end{equation}
Now we turn to estimating the above sequence $\left\{  T^{k}g\right\}  $.
Namely, we prove the following lemma.

\begin{lemma}
\label{lemat 2.3}Let $v$ satisfy $\left(  A1c\right)  $ and $\left(
A2c\right)  .$ Then for $k\in\mathbb{N}_{0},$ $t\in\left[  \alpha
,\beta\right]  $%
\begin{equation}
\left\vert \left(  T^{k}g\right)  \left(  t\right)  \right\vert \leq
\frac{\left(  t-\alpha\right)  ^{k}}{k!}l_{\rho}^{k}M \label{(2.6)}%
\end{equation}
where $l_{\rho}>0$ is a constant defined by $\left(  A2c\right)  $, $M$ is a
constant such that $M=\sup_{t\in\left[  \alpha,\beta\right]  }\left\vert
g\left(  t\right)  \right\vert $ and $\left\{  T^{k}g\right\}  $ is a sequence
defined by $\left(  \ref{(2.4)}\right)  $ and $\left(  \ref{(2.5)}\right)  .$
\end{lemma}

\begin{proof}
The estimate in $\left(  \ref{(2.6)}\right)  $ can be seen by observing that
from $\left(  \ref{(2.3)}\right)  -\left(  \ref{(2.5)}\right)  $ and the
assumptions of the lemma$,$ it follows that
\[
\left\vert \left(  T^{1}g\right)  \left(  t\right)  \right\vert =\left\vert
\int\limits_{\alpha}^{t}v_{x}\left(  t,\tau,x_{0}\left(  \tau\right)  \right)
g\left(  \tau\right)  d\tau\right\vert \leq\left(  t-\alpha\right)  l_{\rho
}M.
\]
Similarly, we have%
\[
\left\vert \left(  T^{2}g\right)  \left(  t\right)  \right\vert =\left\vert
\int\limits_{\alpha}^{t}v_{x}\left(  t,\tau,x_{0}\left(  \tau\right)  \right)
\left(  T^{1}g\right)  \left(  \tau\right)  d\tau\right\vert \leq
\int\limits_{\alpha}^{t}l_{\rho}\left(  \tau-\alpha\right)  l_{\rho}%
Md\tau=\frac{1}{2}\left(  t-\alpha\right)  ^{2}l_{\rho}^{2}M,
\]
and subsequently we get
\[
\left\vert \left(  T^{3}g\right)  \left(  t\right)  \right\vert =\left\vert
\int\limits_{\alpha}^{t}v_{x}\left(  t,\tau,x_{0}\left(  \tau\right)  \right)
\left(  T^{2}g\right)  \left(  \tau\right)  d\tau\right\vert \leq
\int\limits_{\alpha}^{t}l_{\rho}\frac{\left(  \tau-\alpha\right)  ^{2}}%
{2}l_{\rho}^{2}Md\tau=\frac{\left(  t-\alpha\right)  ^{3}}{3!}l_{\rho}^{3}M.
\]
It is enough to proceed by induction on $k$ to obtain the estimate $\left(
\ref{(2.6)}\right)  .$
\end{proof}

We can now formulate the problem to which the rest of this section is
dedicated. For any $t\in\left[  \alpha,\beta\right]  ,$ let us consider the
linear integral equation of the form%
\begin{equation}
h\left(  t\right)  +\int\limits_{\alpha}^{t}v_{x}\left(  t,\tau,x_{0}\left(
\tau\right)  \right)  h\left(  \tau\right)  d\tau=g\left(  t\right)
\label{(2.7)}%
\end{equation}
where $x_{0}\in AC_{0}^{2}$ and $g\in AC_{0}^{2}$ are arbitrarily fixed$.$
Next, we prove the existence and uniqueness results for the above equation.
The lemma to be proved is the following.

\begin{lemma}
\label{lemat2.4}For any $x_{0},g\in AC_{0}^{2},$ the equation $\left(
\ref{(2.7)}\right)  $ possesses a unique solution in $AC_{0}^{2}$ provided
that conditions $\left(  A1\right)  $ and $\left(  A2\right)  $ are
satisfied$.$
\end{lemma}

\begin{proof}
The equation $\left(  \ref{(2.7)}\right)  $ can be rewritten in the form%
\begin{equation}
h+Th=g \label{(2.7')}%
\end{equation}
where
\[
\left(  Th\right)  \left(  t\right)  =\int\limits_{\alpha}^{t}v_{x}\left(
t,\tau,x_{0}\left(  \tau\right)  \right)  h\left(  \tau\right)  d\tau.
\]
Consider the following sequence%
\begin{align*}
h_{k+1}  &  =g-Th_{k},\text{ for }k\in\mathbb{N}_{0},\text{ }\\
h_{0}  &  =0.
\end{align*}
It is easy to observe that
\begin{align}
h_{k+1}  &  =g-Th_{k}=g-Tg+T^{2}g-T^{3}g+...+\left(  -1\right)  ^{k}%
T^{k}g\label{(2.9)}\\
&  =g+\sum\limits_{i=1}^{k}\left(  -1\right)  ^{i}T^{i}g,\text{ for }%
k\in\mathbb{N}_{0},\nonumber
\end{align}
where%
\begin{align}
T^{k+1}g  &  =T\left(  T^{k}g\right)  ,\text{ for }k\in\mathbb{N}%
_{0},\label{(2.10)}\\
T^{0}g  &  =g\nonumber
\end{align}
and $T^{k}g$ is given by the formula in $\left(  \ref{(2.4)}\right)  $. By
assumptions $\left(  A1\right)  $, $\left(  A2\right)  $ and from Lemma
\ref{lemat 2.3}, analysis similar to that in the proof of Lemma \ref{lemat2.1}
leads, if we apply induction on $k$, to the fact that $T^{k}g\in AC_{0}^{2}$
for any $g\in AC_{0}^{2}$ and $k\in\mathbb{N}.$ From Lemma \ref{lemat 2.3},
for $k\in\mathbb{N}$, we get%
\begin{align*}
\left\Vert T^{k}g\right\Vert _{AC_{0}^{2}}^{2}  &  =\int\limits_{\alpha
}^{\beta}\left\vert v_{x}\left(  t,t,x_{0}\left(  t\right)  \right)  \left(
T^{k-1}g\right)  \left(  t\right)  +\int\limits_{\alpha}^{t}v_{xt}\left(
t,\tau,x_{0}\left(  \tau\right)  \right)  \left(  T^{k-1}g\right)  \left(
\tau\right)  d\tau\right\vert ^{2}dt\\
&  \leq\int\limits_{\alpha}^{\beta}\left[  l_{\rho}\frac{\left(  \beta
-\alpha\right)  ^{k-1}}{\left(  k-1\right)  !}l_{\rho}^{k-1}M+\left(
\beta-\alpha\right)  l_{\rho}\frac{\left(  \beta-\alpha\right)  ^{k-1}%
}{\left(  k-1\right)  !}l_{\rho}^{k-1}M\right]  ^{2}dt\\
&  =l_{\rho}^{2k}\left[  \frac{\left(  \beta-\alpha\right)  ^{k-1}}{\left(
k-1\right)  !}\right]  ^{2}M^{2}C^{2}%
\end{align*}
where $C^{2}=\left(  \beta-\alpha\right)  \left(  1+\beta-\alpha\right)
^{2}.$ Consequently, for $k\in\mathbb{N}$, we have%
\begin{equation}
\left\Vert T^{k}g\right\Vert _{AC_{0}^{2}}\leq CMl_{\rho}\frac{l_{\rho}%
^{k-1}\left(  \beta-\alpha\right)  ^{k-1}}{\left(  k-1\right)  !}%
=D\frac{A^{k-1}}{\left(  k-1\right)  !} \label{(2.11)}%
\end{equation}
where $D=CMl_{\rho}>0$ and $A=l_{\rho}\left(  \beta-\alpha\right)  >0.$ Since
the series $\sum\limits_{k=1}^{\infty}D\frac{A^{k-1}}{\left(  k-1\right)  !}$
is a convergent bound of the series $\sum\limits_{k=1}^{\infty}\left(
-1\right)  ^{k}T^{k}g$, it follows that the sequence $\left\{  h_{k+1}%
\right\}  $ defined by $\left(  \ref{(2.9)}\right)  $ and $\left(
\ref{(2.10)}\right)  $ is a Cauchy sequence in $AC_{0}^{2}.$ Therefore, the
sequence $\left\{  h_{k}\right\}  $ tends to some $h_{0}\in AC_{0}^{2}.$
Clearly, the operator $T:AC_{0}^{2}$ $\rightarrow AC_{0}^{2}$ defined by
$\left(  \ref{(2.4)}\right)  $ is continuous. Hence, $h_{0}+Th_{0}=g,$ i.e.
$h_{0}$ is a solution to the equation $\left(  \ref{(2.7)}\right)  .$\newline
What is left is to prove that $h_{0}$ is unique. Suppose, on the contrary,
that there exists $h_{1}\in AC_{0}^{2}$ such that $h_{1}\neq h_{0}$ and
satisfies the equation $\left(  \ref{(2.7)}\right)  .$ Denote $\tilde{h}%
=h_{1}-h_{0}.$ From $\left(  \ref{(2.7')}\right)  $ we see that%
\[
\tilde{h}+T\tilde{h}=0.
\]
Moreover,%
\begin{align*}
0  &  =T\tilde{h}+T^{2}\tilde{h}=-\tilde{h}+T^{2}\tilde{h},\\
0  &  =-T\tilde{h}+T^{3}\tilde{h}=\tilde{h}+T^{3}\tilde{h},\\
&  ...,\\
0  &  =\left(  -1\right)  ^{k+1}\tilde{h}+T^{k}\tilde{h}\text{ for }%
k\in\mathbb{N}.
\end{align*}
On making use of the estimate $\left(  \ref{(2.11)}\right)  $ we conclude that
$T^{k}\tilde{h}\rightarrow0$ as $k\rightarrow\infty$ and hence $\tilde
{h}=h_{1}-h_{0}=0.$ Therefore, the equation $\left(  \ref{(2.7)}\right)  $
possesses exactly one solution in $AC_{0}^{2},$ and the proof of the lemma is
complete$.$
\end{proof}

\section{Palais-Smale condition for the functional defined by the Volterra
type operator}

Let us consider the functional $F_{y}:AC_{0}^{2}\rightarrow\mathbb{R}^{+}$ of
the form%
\begin{equation}
F_{y}\left(  x\right)  =\frac{1}{2}\left\Vert V\left(  x\right)  -y\right\Vert
_{AC_{0}^{2}}^{2}=\frac{1}{2}\int\limits_{\alpha}^{\beta}\left\vert \frac
{d}{dt}V\left(  x\right)  \left(  t\right)  -y^{\prime}\left(  t\right)
\right\vert ^{2}dt \label{coercive}%
\end{equation}
where the operator $V$ is defined by $\left(  \ref{0}\right)  $ and $y\in
AC_{0}^{2}$ is a given function.

Differentiating the operator $V$ and substituting the result into $\left(
\ref{coercive}\right)  ,$ we obtain
\begin{equation}
F_{y}\left(  x\right)  =\frac{1}{2}\int\limits_{\alpha}^{\beta}\left\vert
x^{\prime}(t)+v\left(  t,t,x\left(  t\right)  \right)  +\int\limits_{\alpha
}^{t}v_{t}\left(  t,\tau,x\left(  \tau\right)  \right)  d\tau-y^{\prime
}\left(  t\right)  \right\vert ^{2}dt. \label{(3.3)}%
\end{equation}
The task is now to obtain conditions under which, for any $y\in AC_{0}^{2},$
the functional $F_{y}$ satisfies Palais-Smale condition. We recall what this
means. A sequence $\left\{  x_{k}\right\}  _{k\in\mathbb{N}}$ is referred to
as a Palais-Smale sequence for a functional $\varphi$ if for some $M>0$, any
$k\in\mathbb{N},$ $\left\vert \varphi\left(  x_{k}\right)  \right\vert \leq M$
and $\varphi^{\prime}\left(  x_{k}\right)  \rightarrow0$ as $k\rightarrow
\infty.$ We say that $\varphi$ satisfies Palais-Smale condition if any
Palais-Smale sequence possesses a convergent subsequence. To accomplish the
task of guaranteeing Palais-Smale condition, we find ourselves forced to
introduce an extra assumption. First, we consider, for simplicity without loss
of generality, the case of the function $v$ vanishing on the diagonal of the
square $\left[  \alpha,\beta\right]  \times\left[  \alpha,\beta\right]  .$
This case encompasses the problems with the integral operators with kernels
depending on the difference of arguments $t$ and $\tau,$ i.e. $t-\tau,$ with
value $0$ at $0,$ and the space variable $x.$ This kind of operators appear
very often in applications, for a survey of what is known up to date, see, for
example, \cite{Chr, Gri, RenHruNoh} and references therein.

On $v$ we shall impose the following condition:

\begin{enumerate}
\item[(A3)] (a) $v(t,t,x)=0$ for any $t\in\left[  \alpha,\beta\right]  $,
$x\in\mathbb{R}^{n},$\newline$\ \ \ \ \ \ \ $(b) $\left\vert v_{t}\left(
t,\tau,x\right)  \right\vert \leq c_{0}\left(  t,\tau\right)  \left\vert
x\right\vert +d_{0}\left(  t,\tau\right)  $ where $\left(  t,\tau\right)  \in
P_{\Delta},$ $c_{0},d_{0}\in L^{2}\left(  P_{\Delta},\mathbb{R}^{+}\right)  ,$
\newline\ \ \ \ \ \ (c) $\left\Vert c_{0}\right\Vert _{L^{2}\left(  P_{\Delta
},\mathbb{R}^{+}\right)  }<\frac{\sqrt{2}}{2\left(  \beta-\alpha\right)  }.$
\end{enumerate}

It transpires that it is possible to show that, for any $y\in AC_{0}^{2},$
under assumptions imposed in $\left(  A1\right)  ,$ $\left(  A2\right)  ,$
$\left(  A3\right)  $, the functional $F_{y}$ is coercive, i.e. for any $y\in
AC_{0}^{2},$ $F_{y}\left(  x\right)  \rightarrow\infty$ as $\left\Vert
x\right\Vert _{AC_{0}^{2}}\rightarrow\infty.$ Actually we have the following lemma.

\begin{lemma}
\label{lemat3.2}If the function $v$ satisfies $\left(  A1a\right)  ,$ $\left(
A1b\right)  ,$ $\left(  A2a\right)  ,$ $\left(  A2b\right)  $ and $\left(
A3\right)  $, then for any $y\in AC_{0}^{2}$ the functional $F_{y}$ defined by
$\left(  \ref{coercive}\right)  $ is coercive$.$
\end{lemma}

\begin{proof}
Let us observe that the functional $F_{0}$ is bounded from below, since for
any $y\in AC_{0}^{2}$ the functional $F_{y}$ is coercive if and only if the
functional $F_{0}$ is coercive. Indeed, by the Schwarz inequality and the
assumptions of the lemma, we obtain
\begin{align*}
F_{0}\left(  x\right)   &  =\frac{1}{2}\int\limits_{\alpha}^{\beta}\left\vert
x^{\prime}\left(  t\right)  +\int\limits_{\alpha}^{t}v_{t}\left(
t,\tau,x\left(  \tau\right)  \right)  d\tau\right\vert ^{2}dt\\
&  \geq\frac{1}{2}\left\Vert x\right\Vert _{AC_{0}^{2}}^{2}-\int
\limits_{\alpha}^{\beta}\left\vert x^{\prime}\left(  t\right)  \right\vert
\left\vert \int\limits_{\alpha}^{t}v_{t}\left(  t,\tau,x\left(  \tau\right)
\right)  d\tau\right\vert dt\\
&  \geq\left\Vert x\right\Vert _{AC_{0}^{2}}^{2}\left(  \frac{1}{2}-\frac
{1}{\sqrt{2}}\left\Vert c_{0}\right\Vert _{L^{2}\left(  P_{\Delta}%
,\mathbb{R}^{+}\right)  }\left(  \beta-\alpha\right)  \right)  -\left\Vert
x\right\Vert _{AC_{0}^{2}}\left(  \beta-\alpha\right)  ^{\frac{1}{2}%
}\left\Vert d_{0}\right\Vert _{L^{2}\left(  P_{\Delta},\mathbb{R}^{+}\right)
}.
\end{align*}
From $\left(  A3c\right)  $ and the above estimate it follows that
$F_{0}\left(  x\right)  \rightarrow\infty$ if $\left\Vert x\right\Vert
_{AC_{0}^{2}}\rightarrow\infty.$ Consequently, for any $y\in AC_{0}^{2},$
$F_{y}\left(  x\right)  \rightarrow\infty$ if $\left\Vert x\right\Vert
_{AC_{0}^{2}}\rightarrow\infty.$
\end{proof}

Apart from the function $v$ vanishing on the diagonal of the square $\left[
\alpha,\beta\right]  \times\left[  \alpha,\beta\right]  ,$ we shall consider
the case of the function $v$ satisfying the following growth condition:

\begin{enumerate}
\item[(A4)] (a) $\left\vert v\left(  t,t,x\right)  \right\vert \leq
c_{1}\left(  t\right)  \left\vert x\right\vert +d_{1}\left(  t\right)  $ where
$t\in\left[  \alpha,\beta\right]  ,$ $c_{1},d_{1}\in L^{2}\left(  \left[
\alpha,\beta\right]  ,\mathbb{R}^{+}\right)  ,$\newline\ \ \ \ \ \ (b)
$\left\vert v_{t}\left(  t,\tau,x\right)  \right\vert \leq c_{2}\left(
t,\tau\right)  \left\vert x\right\vert +d_{2}\left(  t,\tau\right)  $ where
$\left(  t,\tau\right)  \in P_{\Delta},$ $c_{2},d_{2}\in L^{2}\left(
P_{\Delta},\mathbb{R}^{+}\right)  ,$\newline\ \ \ \ \ \ (c) $\left\Vert
\tilde{c}\right\Vert _{L^{2}\left(  \left[  \alpha,\beta\right]
,\mathbb{R}^{+}\right)  }<\frac{1}{2}$ where $\tilde{c}\left(  t\right)
=\left(  t-\alpha\right)  ^{\frac{1}{2}}c_{1}\left(  t\right)  +\frac{1}%
{\sqrt{2}}\left(  \beta-a\right)  \left(  \int_{\alpha}^{t}c_{2}^{2}\left(
t,\tau\right)  d\tau\right)  ^{\frac{1}{2}}$ for any $t\in\left[  \alpha
,\beta\right]  .$
\end{enumerate}

If in Lemma \ref{lemat3.2} instead of $\left(  A3\right)  $ we assume that
$\left(  A4\right)  $ is fulfilled, then the conclusion of the Lemma
\ref{lemat3.2} still holds. Hence, the lemma to be proved now is the
following.

\begin{lemma}
\label{lemat3.1}If the function $v$ satisfies $\left(  A1a\right)  ,$ $\left(
A1b\right)  ,$ $(A2a),$ $(A2b)$ and $\left(  A4\right)  $, then for any $y\in
AC_{0}^{2}$ the functional $F_{y}$ defined by $\left(  \ref{coercive}\right)
$ is coercive$.$
\end{lemma}

\begin{proof}
For any $t\in\left[  \alpha,\beta\right]  ,$ let us first examine the
function
\[
\tilde{v}\left(  t,x\left(  t\right)  \right)  =v\left(  t,t,x\left(
t\right)  \right)  +\int\limits_{\alpha}^{t}v_{t}\left(  t,\tau,x\left(
\tau\right)  \right)  d\tau.
\]
By the Schwarz inequality and inequality $\left(  \ref{((2.1))}\right)  ,$ we
get for any $t\in\left[  \alpha,\beta\right]  $%
\begin{align}
\left\vert \tilde{v}\left(  t,x\left(  t\right)  \right)  \right\vert  &
\leq\left\vert v\left(  t,t,x\left(  t\right)  \right)  \right\vert
+\int\limits_{\alpha}^{t}\left\vert v_{t}\left(  t,\tau,x\left(  \tau\right)
\right)  \right\vert d\tau\label{(3.4')}\\
&  \leq\left\Vert x\right\Vert _{AC_{0}^{2}}\left[  c_{1}\left(  t\right)
\left(  t-\alpha\right)  ^{\frac{1}{2}}+\left(  \int\limits_{\alpha}^{t}%
c_{2}^{2}\left(  t,\tau\right)  d\tau\right)  ^{\frac{1}{2}}\frac{1}{\sqrt{2}%
}\left(  \beta-\alpha\right)  \right]  +d_{1}\left(  t\right)  +\int_{\alpha
}^{t}d_{2}\left(  t,\tau\right)  d\tau\nonumber\\
&  =\tilde{c}\left(  t\right)  \left\Vert x\right\Vert _{AC_{0}^{2}}+\tilde
{d}\left(  t\right)  \nonumber
\end{align}
where $\tilde{c}\left(  t\right)  $ is defined in $\left(  A4\right)  $ and
$\tilde{d}\left(  t\right)  =d_{1}\left(  t\right)  +\int_{\alpha}^{t}%
d_{2}\left(  t,\tau\right)  d\tau.$ Proceeding as in the proof of Lemma
\ref{lemat3.2}, let us show that the functional $F_{0}$ is coercive$.$ We see
from $\left(  \ref{(3.4')}\right)  $ that%
\begin{align*}
F_{0}\left(  x\right)   &  =\frac{1}{2}\int\limits_{\alpha}^{\beta}\left\vert
x^{\prime}\left(  t\right)  +\tilde{v}\left(  t,x\left(  t\right)  \right)
\right\vert ^{2}dt\\
&  \geq\frac{1}{2}\left\Vert x\right\Vert _{AC_{0}^{2}}^{2}-\int
\limits_{\alpha}^{\beta}\left\vert x^{\prime}\left(  t\right)  \right\vert
\left\vert \tilde{v}\left(  t,x\left(  t\right)  \right)  \right\vert dt\\
&  \geq\left(  \frac{1}{2}-\left\Vert \tilde{c}\right\Vert _{L^{2}\left(
\left[  \alpha,\beta\right]  ,\mathbb{R}^{+}\right)  }\right)  \left\Vert
x\right\Vert _{AC_{0}^{2}}^{2}-\left\Vert \tilde{d}\right\Vert _{L^{2}\left(
\left[  \alpha,\beta\right]  ,\mathbb{R}^{+}\right)  }\left\Vert x\right\Vert
_{AC_{0}^{2}}.
\end{align*}
Consequently, from $\left(  A4c\right)  $ it follows that $F_{0}\left(
x\right)  \rightarrow\infty$ if $\left\Vert x\right\Vert _{AC_{0}^{2}%
}\rightarrow\infty.$
\end{proof}

In the next lemmas we provide some conditions under which, for any $y\in
AC_{0}^{2},$ the functional $F_{y}$ defined by $\left(  \ref{coercive}\right)
$ satisfies Palais-Smale condition$.$

\begin{lemma}
\label{lemat3.3}Suppose that conditions (A1), (A2) and (A3) hold. Then for any
$y\in AC_{0}^{2}$ the functional $F_{y}$ defined by $\left(  \ref{coercive}%
\right)  $ satisfies Palais-Smale condition.
\end{lemma}

\begin{proof}
By $\left(  A3a\right)  ,$ $v(t,t,x)=0$ for any $t\in\left[  \alpha
,\beta\right]  $ and consequently, for any $y\in AC_{0}^{2},$ the functional
$F_{y}$ has the form%
\begin{equation}
F_{y}\left(  x\right)  =\frac{1}{2}\int\limits_{\alpha}^{\beta}\left\vert
x^{\prime}\left(  t\right)  +\int\limits_{\alpha}^{t}v_{t}\left(
t,\tau,x\left(  \tau\right)  \right)  d\tau-y^{\prime}\left(  t\right)
\right\vert ^{2}dt. \label{(3.6)}%
\end{equation}
It is worth noting that the functional $F_{y}$ being a superposition of two
$C^{1}-$mappings is also of the same regularity type. Moreover, substituting
$x=z+y$ into $\left(  \ref{(3.6)}\right)  $ we get%
\begin{equation}
\tilde{F}\left(  z\right)  =F_{y}\left(  z+y\right)  =\frac{1}{2}%
\int\limits_{\alpha}^{\beta}\left\vert z^{\prime}\left(  t\right)
+\int\limits_{\alpha}^{t}v_{t}\left(  t,\tau,z\left(  \tau\right)  +y\left(
\tau\right)  \right)  d\tau\right\vert ^{2}dt. \label{(3.7)}%
\end{equation}
Undeniably, for any $y\in AC_{0}^{2},$ the functional $F_{y}$ satisfies
Palais-Smale condition if and only if $\tilde{F}$ satisfies Palais-Smale
condition. Immediately, from $\left(  \ref{(3.7)}\right)  $ we get%
\begin{equation}
\tilde{F}\left(  z\right)  =\frac{1}{2}\int\limits_{\alpha}^{\beta}\left[
\left\vert z^{\prime}\left(  t\right)  \right\vert ^{2}+2\left\langle
z^{\prime}\left(  t\right)  ,\int\limits_{\alpha}^{t}v_{t}\left(
t,\tau,z\left(  \tau\right)  +y\left(  \tau\right)  \right)  d\tau
\right\rangle +\left\vert \int\limits_{\alpha}^{t}v_{t}\left(  t,\tau,z\left(
\tau\right)  +y\left(  \tau\right)  \right)  d\tau\right\vert ^{2}\right]  dt
\label{(3.8)}%
\end{equation}
and the differential of $\tilde{F}$ at $z\in AC_{0}^{2}$ is given, for any
$h\in AC_{0}^{2},$ by
\begin{align}
\tilde{F}^{\prime}\left(  z\right)  h  &  =\int\limits_{\alpha}^{\beta}\left[
\left\langle z^{\prime}\left(  t\right)  ,h^{\prime}\left(  t\right)
\right\rangle +\left\langle h^{\prime}\left(  t\right)  ,\int\limits_{\alpha
}^{t}v_{t}\left(  t,\tau,z\left(  \tau\right)  +y\left(  \tau\right)  \right)
d\tau\right\rangle \right. \label{(3.9)}\\
&  +\left.  \left\langle z^{\prime}\left(  t\right)  ,\int\limits_{\alpha}%
^{t}v_{tx}\left(  t,\tau,z\left(  \tau\right)  +y\left(  \tau\right)  \right)
h\left(  \tau\right)  d\tau\right\rangle \right. \nonumber\\
&  +\left.  \left\langle \int\limits_{\alpha}^{t}v_{t}\left(  t,\tau,z\left(
\tau\right)  +y\left(  \tau\right)  \right)  d\tau,\int\limits_{\alpha}%
^{t}v_{tx}\left(  t,\tau,z\left(  \tau\right)  +y\left(  \tau\right)  \right)
h\left(  \tau\right)  d\tau\right\rangle \right]  dt.\nonumber
\end{align}
Fix $M\geq0$ and let $\left\{  z_{k}\right\}  _{k\in\mathbf{N}}\subset
AC_{0}^{2}$ be a sequence such that $\left\vert \tilde{F}\left(  z_{k}\right)
\right\vert \leq M$ and $\tilde{F}\left(  z_{k}\right)  \rightarrow0$ as
$k\rightarrow\infty.$ From Lemma \ref{lemat3.2} it follows that $\tilde{F}$ is
coercive and hence the sequence $\left\{  z_{k}\right\}  _{k\in\mathbf{N}}$ is
weakly compact.\ Passing, if necessary to a subsequence, we can assume that
$z_{k}$ tends weakly in $AC_{0}^{2}$ to some $z_{0}\in AC_{0}^{2}.$ It may be
worth reminding that the weak convergence of the sequence $\left\{
z_{k}\right\}  _{k\in\mathbf{N}}$ in $AC_{0}^{2}$ implies the uniform
convergence in $C$ and the weak convergence of derivatives in $L^{2}\ $and
moreover a convergent sequence of derivatives has to be bounded. What is left
is to prove that the sequence $\left\{  z_{k}\right\}  _{k\in\mathbf{N}}$
converges to $z_{0}$ in $AC_{0}^{2}.$ Because of $\left(  \ref{(3.9)}\right)
$, direct calculations lead to the equality%
\begin{equation}
\left\langle \tilde{F}^{\prime}\left(  z_{k}\right)  -\tilde{F}^{\prime
}\left(  z_{0}\right)  ,z_{k}-z_{0}\right\rangle =\left\Vert z_{k}%
-z_{0}\right\Vert _{AC_{0}^{2}}^{2}+\sum\limits_{i=1}^{5}G_{i}\left(
z_{k}\right)  \label{(3.10)}%
\end{equation}
where%
\begin{align*}
G_{1}\left(  z_{k}\right)   &  =\int\limits_{\alpha}^{\beta}\left\langle
z_{k}^{\prime}\left(  t\right)  -z_{0}^{\prime}\left(  t\right)
,\int\limits_{\alpha}^{t}\left[  v_{t}\left(  t,\tau,z_{k}\left(  \tau\right)
+y\left(  \tau\right)  \right)  -v_{t}\left(  t,\tau,z_{0}\left(  \tau\right)
+y\left(  \tau\right)  \right)  \right]  d\tau\right\rangle dt,\\
G_{2}\left(  z_{k}\right)   &  =\int\limits_{\alpha}^{\beta}\left\langle
z_{k}^{\prime}\left(  t\right)  ,\int\limits_{\alpha}^{t}v_{tx}\left(
t,\tau,z_{k}\left(  \tau\right)  +y\left(  \tau\right)  \right)  \left(
z_{k}\left(  \tau\right)  -z_{0}\left(  \tau\right)  \right)  d\tau
\right\rangle dt,\\
G_{3}\left(  z_{k}\right)   &  =\int\limits_{\alpha}^{\beta}\left\langle
\int\limits_{\alpha}^{t}v_{t}\left(  t,\tau,z_{k}\left(  \tau\right)
+y\left(  \tau\right)  \right)  d\tau,\int\limits_{\alpha}^{t}v_{tx}\left(
t,\tau,z_{k}\left(  \tau\right)  +y\left(  \tau\right)  \right)  \left(
z_{k}\left(  \tau\right)  -z_{0}\left(  \tau\right)  \right)  d\tau
\right\rangle dt,\\
G_{4}\left(  z_{k}\right)   &  =-\int\limits_{\alpha}^{\beta}\left\langle
z_{0}^{\prime}\left(  t\right)  ,\int\limits_{\alpha}^{t}v_{tx}\left(
t,\tau,z_{0}\left(  \tau\right)  +y\left(  \tau\right)  \right)  \left(
z_{k}\left(  \tau\right)  -z_{0}\left(  \tau\right)  \right)  d\tau
\right\rangle dt,\\
G_{5}\left(  z_{k}\right)   &  =-\int\limits_{\alpha}^{\beta}\left\langle
\int\limits_{\alpha}^{t}v_{t}\left(  t,\tau,z_{0}\left(  \tau\right)
+y\left(  \tau\right)  \right)  d\tau,\int\limits_{\alpha}^{t}v_{tx}\left(
t,\tau,z_{0}\left(  \tau\right)  +y\left(  \tau\right)  \right)  \left(
z_{k}\left(  \tau\right)  -z_{0}\left(  \tau\right)  \right)  d\tau
\right\rangle dt.
\end{align*}
Since $\tilde{F}^{\prime}\left(  z_{k}\right)  \rightarrow0$ and
$z_{k}\rightharpoonup z_{0}$ weakly in $AC_{0}^{2},$ $\lim_{k\rightarrow
\infty}\left\langle \tilde{F}^{\prime}\left(  z_{k}\right)  -\tilde{F}%
^{\prime}\left(  z_{0}\right)  ,z_{k}-z_{0}\right\rangle =0.$ It is therefore
enough to show that $\lim_{k\rightarrow\infty}G_{i}\left(  z_{k}\right)  =0$
for $i=1,2,...,5.$ It can be easily estimated%
\begin{align*}
\left\vert G_{1}\left(  z_{k}\right)  \right\vert ^{2}  &  \leq\int
\limits_{\alpha}^{\beta}\left\vert z_{k}^{\prime}\left(  t\right)
-z_{0}^{\prime}\left(  t\right)  \right\vert ^{2}dt\\
&  \cdot\int\limits_{\alpha}^{\beta}\left[  \int\limits_{\alpha}^{t}\left\vert
v_{t}\left(  t,\tau,z_{k}\left(  \tau\right)  +y\left(  \tau\right)  \right)
-v_{t}\left(  t,\tau,z_{0}\left(  \tau\right)  +y\left(  \tau\right)  \right)
\right\vert d\tau\right]  ^{2}dt.
\end{align*}
Regarding the above inequality, it is clear that the first factor is bounded
whereas the second one converges to zero, which is an immediate consequence of
the Lebesgue Theorem, and therefore $G_{1}\left(  z_{k}\right)  \rightarrow0$
as $k\rightarrow0.$ A similar argument holds for the other terms, i.e. by the
Schwarz inequality and the uniform convergence of $\left\{  z_{k}\right\}  $
to $z_{0}$ can be demonstrated that $G_{i}\left(  z_{k}\right)  \rightarrow0$
if $k\rightarrow\infty$ for $i=2,3,4,5.$ Consequently, from $\left(
\ref{(3.10)}\right)  $ it follows that $z_{k}\rightarrow z_{0}$ in $AC_{0}%
^{2},$ which completes the proof.
\end{proof}

Furthermore, we have the following lemma.

\begin{lemma}
\label{lemat3.4}If assumptions (A1), (A2) and (A4) hold then for any $y\in
AC_{0}^{2}$ the functional $F_{y}$ defined by $\left(  \ref{coercive}\right)
$ satisfies Palais-Smale condition.
\end{lemma}

\begin{proof}
The proof proceeds along the same lines as the proof of Lemma \ref{lemat3.3},
this time the details are more intricate. Clearly, the functional $F_{y}$ has
the form%
\begin{equation}
F_{y}\left(  x\right)  =\frac{1}{2}\int\limits_{\alpha}^{\beta}\left\vert
x^{\prime}\left(  t\right)  +v\left(  t,t,x\left(  t\right)  \right)
+\int\limits_{\alpha}^{t}v_{t}\left(  t,\tau,x\left(  \tau\right)  \right)
d\tau-y^{\prime}\left(  t\right)  \right\vert ^{2}dt. \label{nowy}%
\end{equation}
Subsequently, in a similar fashion as in the proof of Lemma \ref{lemat3.3}, we
obtain%
\begin{align*}
\bar{F}\left(  z\right)   &  =F_{y}\left(  z+y\right)  =\frac{1}{2}%
\int\limits_{\alpha}^{\beta}\left\vert z^{\prime}\left(  t\right)  +v\left(
t,t,z\left(  t\right)  +y\left(  t\right)  \right)  +\int\limits_{\alpha}%
^{t}v_{t}\left(  t,\tau,z\left(  \tau\right)  +y\left(  \tau\right)  \right)
d\tau\right\vert ^{2}dt\\
&  =\tilde{F}\left(  z\right)  +\int\limits_{\alpha}^{\beta}\left(
\left\langle z^{\prime}\left(  t\right)  ,v\left(  t,t,z\left(  t\right)
+y\left(  t\right)  \right)  \right\rangle +\frac{1}{2}\left\vert v\left(
t,t,z\left(  t\right)  +y\left(  t\right)  \right)  \right\vert ^{2}\right. \\
&  +\left.  \left\langle v\left(  t,t,z\left(  t\right)  +y\left(  t\right)
\right)  ,\int\limits_{\alpha}^{t}v_{t}\left(  t,\tau,z\left(  \tau\right)
+y\left(  \tau\right)  \right)  d\tau\right\rangle \right)  dt
\end{align*}
where $\tilde{F}$ is defined in $\left(  \ref{(3.8)}\right)  $. Immediately
from the above, for any $h\in AC_{0}^{2}$ we obtain%
\begin{align*}
\bar{F}^{\prime}\left(  z\right)  h  &  =\tilde{F}^{\prime}\left(  z\right)
h+\int\limits_{\alpha}^{\beta}\left(  \left\langle h^{\prime}\left(  t\right)
,v\left(  t,t,z\left(  t\right)  +y\left(  t\right)  \right)  \right\rangle
+\left\langle z^{\prime}\left(  t\right)  ,v_{x}\left(  t,t,z\left(  t\right)
+y\left(  t\right)  \right)  h\left(  t\right)  \right\rangle \right. \\
&  +\left.  \left\langle v\left(  t,t,z\left(  t\right)  +y\left(  t\right)
\right)  ,v_{x}\left(  t,t,z\left(  t\right)  +y\left(  t\right)  \right)
h\left(  t\right)  \right\rangle \right. \\
&  +\left.  \left\langle v_{x}\left(  t,t,z\left(  t\right)  +y\left(
t\right)  \right)  h\left(  t\right)  ,\int\limits_{\alpha}^{t}v_{t}\left(
t,\tau,z\left(  \tau\right)  +y\left(  \tau\right)  \right)  d\tau
\right\rangle \right. \\
&  +\left.  \left\langle v\left(  t,t,z\left(  t\right)  +y\left(  t\right)
\right)  ,\int\limits_{\alpha}^{t}v_{tx}\left(  t,\tau,z\left(  \tau\right)
+y\left(  \tau\right)  \right)  h\left(  \tau\right)  d\tau\right\rangle
\right)  dt
\end{align*}
and
\[
\left\langle \bar{F}^{\prime}\left(  z_{k}\right)  -\bar{F}^{\prime}\left(
z_{0}\right)  ,z_{k}-z_{0}\right\rangle =\left\Vert z_{k}-z_{0}\right\Vert
_{AC_{0}^{2}}^{2}+\sum\limits_{i=1}^{5}G_{i}\left(  z_{k}\right)
+\sum\limits_{i=1}^{9}H_{i}\left(  z_{k}\right)
\]
where%
\begin{align*}
H_{1}\left(  z_{k}\right)   &  =\int\limits_{\alpha}^{\beta}\left\langle
z_{k}^{\prime}\left(  t\right)  -z_{0}^{\prime}\left(  t\right)  ,v\left(
t,t,z_{k}\left(  t\right)  +y\left(  t\right)  \right)  -v\left(
t,t,z_{0}\left(  t\right)  +y\left(  t\right)  \right)  \right\rangle dt,\\
H_{2}\left(  z_{k}\right)   &  =\int\limits_{\alpha}^{\beta}\left\langle
z_{k}^{\prime}\left(  t\right)  ,v_{x}\left(  t,t,z_{k}\left(  t\right)
+y\left(  t\right)  \right)  \left(  z_{k}\left(  t\right)  -z_{0}\left(
t\right)  \right)  \right\rangle dt,\\
H_{3}\left(  z_{k}\right)   &  =\int\limits_{\alpha}^{\beta}\left\langle
v\left(  t,t,z_{k}\left(  t\right)  +y\left(  t\right)  \right)  ,v_{x}\left(
t,t,z_{k}\left(  t\right)  +y\left(  t\right)  \right)  \left(  z_{k}\left(
t\right)  -z_{0}\left(  t\right)  \right)  \right\rangle dt,
\end{align*}%
\begin{align*}
H_{4}\left(  z_{k}\right)   &  =\int\limits_{\alpha}^{\beta}\left\langle
v_{x}\left(  t,t,z_{k}\left(  t\right)  +y\left(  t\right)  \right)  \left(
z_{k}\left(  t\right)  -z_{0}\left(  t\right)  \right)  ,\int\limits_{\alpha
}^{t}v_{t}\left(  t,\tau,z_{k}\left(  \tau\right)  +y\left(  \tau\right)
\right)  d\tau\right\rangle dt,\\
H_{5}\left(  z_{k}\right)   &  =\int\limits_{\alpha}^{\beta}\left\langle
v\left(  t,t,z_{k}\left(  t\right)  +y\left(  t\right)  \right)
,\int\limits_{\alpha}^{t}v_{tx}\left(  t,\tau,z_{k}\left(  \tau\right)
+y\left(  \tau\right)  \right)  \left(  z_{k}\left(  \tau\right)
-z_{0}\left(  \tau\right)  \right)  d\tau\right\rangle dt,\\
H_{6}\left(  z_{k}\right)   &  =-\int\limits_{\alpha}^{\beta}\left\langle
z_{0}^{\prime}\left(  t\right)  ,v_{x}\left(  t,t,z_{0}\left(  t\right)
+y\left(  t\right)  \right)  \left(  z_{k}\left(  t\right)  -z_{0}\left(
t\right)  \right)  \right\rangle dt,\\
H_{7}\left(  z_{k}\right)   &  =-\int\limits_{\alpha}^{\beta}\left\langle
v\left(  t,t,z_{0}\left(  t\right)  +y\left(  t\right)  \right)  ,v_{x}\left(
t,t,z_{0}\left(  t\right)  +y\left(  t\right)  \right)  \left(  z_{k}\left(
t\right)  -z_{0}\left(  t\right)  \right)  \right\rangle dt,\\
H_{8}\left(  z_{k}\right)   &  =-\int\limits_{\alpha}^{\beta}\left\langle
v_{x}\left(  t,t,z_{0}\left(  t\right)  +y\left(  t\right)  \right)  \left(
z_{k}\left(  t\right)  -z_{0}\left(  t\right)  \right)  ,\int\limits_{\alpha
}^{t}v_{t}\left(  t,\tau,z_{0}\left(  \tau\right)  +y\left(  \tau\right)
\right)  d\tau\right\rangle dt,\\
H_{9}\left(  z_{k}\right)   &  =-\int\limits_{\alpha}^{\beta}\left\langle
v\left(  t,t,z_{0}\left(  t\right)  +y\left(  t\right)  \right)
,\int\limits_{\alpha}^{t}v_{tx}\left(  t,\tau,z_{0}\left(  \tau\right)
+y\left(  \tau\right)  \right)  \left(  z_{k}\left(  \tau\right)
-z_{0}\left(  \tau\right)  \right)  d\tau\right\rangle dt.
\end{align*}
The rest of the proof runs as in the proof of Lemma \ref{lemat3.3}.
\end{proof}


\section{Main results}

We begin with the statement of the theorem on a diffeomorphism between Banach
and Hilbert spaces.

\begin{theorem}
\label{Tw4.1}Let $X$ be a real Banach space, $H$ be a real Hilbert space,
$V:X\rightarrow H$ be an operator of $C^{1}$ class. If \newline(a1) for any
$x\in X,$ the equation $V^{\prime}\left(  x\right)  h=g$ possesses a unique
solution for any $g\in H,$\newline(a2) for any $y\in H,$ the functional
$F_{y}\left(  x\right)  =\frac{1}{2}\left\Vert V\left(  x\right)
-y\right\Vert _{H}^{2}$ satisfies Palais-Smale condition$,$ \newline then $V$
is a diffeomorphism$.$
\end{theorem}

\begin{remark}
From the bounded inverse theorem it follows that for any $x\in X$ there exists
a constant $\alpha_{x}>0$ such that $\left\Vert V^{\prime}\left(  x\right)
h\right\Vert _{H}\geq\alpha_{x}\left\Vert h\right\Vert _{X}$. Therefore, it
can be easily seen that, in notation $F_{y}=\varphi$ and $V=f,$ the above
theorem is equivalent to Theorem 3.1 in \cite{IdcSkoWal}.
\end{remark}

The application of Lemmas \ref{lemat2.4}, \ref{lemat3.3}, \ref{lemat3.4} and
Theorem \ref{Tw4.1} leads to the main conclusion of this paper.

\begin{theorem}
\label{Tw4.2}If the function $v$ satisfies assumptions (A1), (A2) and one of
the assumption either (A3) or (A4) holds, then the integral operator
$V:AC_{0}^{2}\rightarrow AC_{0}^{2}$ defined by $\left(  \ref{0}\right)  $ is
a diffeomorphism$.$
\end{theorem}

\begin{proof}
Set $X=H=AC_{0}^{2}.$ From Lemma \ref{lemat2.4} we infer that the operator $V$
satisfies the assumption (a1) of Theorem \ref{Tw4.1}, while either Lemma
\ref{lemat3.3} or Lemma \ref{lemat3.4} ascertains that for any $y\in
AC_{0}^{2}$ the functional\ $F_{y}\left(  x\right)  =\frac{1}{2}\left\Vert
V\left(  x\right)  -y\right\Vert _{AC_{0}^{2}}^{2}$ satisfies Palais-Smale
condition, i.e. the assumption (a2) of Theorem \ref{Tw4.1} is fulfilled.
Therefore,\ in result $V:AC_{0}^{2}\rightarrow AC_{0}^{2}$ is a
diffeomorphism$.$
\end{proof}

Theorem \ref{Tw4.2} can be restated as follows.

\begin{theorem}
\label{Tw4.2'}If the function $v$ satisfies the assumptions of Theorem
\ref{Tw4.2}, then for any $a\in AC_{0}^{2}$ the nonlinear integral equation
\[
x\left(  t\right)  +\int\limits_{\alpha}^{t}v\left(  t,\tau,x\left(
\tau\right)  \right)  d\tau=a\left(  t\right)
\]
possesses a unique solution $x=x_{a}\in AC_{0}^{2}.$ The solution $x=x_{a}$
depends continuously on parameter $a$ and moreover the operator $AC_{0}%
^{2}\backepsilon a\rightarrow x_{a}\in AC_{0}^{2}$ is continuously Fr\'{e}chet differentiable.
\end{theorem}


\section{Examples}

\begin{example}
As an example of the application of Theorem \ref{Tw4.2}, consider the
following integral operator for any $t\in\left[  0,1\right]  $%
\[
V\left(  x\right)  \left(  t\right)  =x\left(  t\right)  +\bar{a}%
\int\limits_{0}^{t}\left(  t-\tau\right)  ^{\frac{2}{3}}\ln\left(  1+2\left(
t-\tau\right)  ^{2}x^{2}\left(  \tau\right)  \right)  d\tau\text{ }%
\]
with $\bar{a}^{2}<\frac{35}{8}$, i.e the function $v:P_{\Delta}\times
\mathbb{R}\rightarrow\mathbb{R}$ is of the form%
\[
v\left(  t,\tau,x\right)  =\bar{a}\left(  t-\tau\right)  ^{\frac{2}{3}}%
\ln\left(  1+2\left(  t-\tau\right)  ^{2}x^{2}\right)  .
\]
Let us notice%
\[
v_{t}\left(  t,\tau,x\right)  =\frac{2}{3}\bar{a}\left(  t-\tau\right)
^{-\frac{1}{3}}\ln\left(  1+2\left(  t-\tau\right)  ^{2}x^{2}\right)  +\bar
{a}\left(  t-\tau\right)  ^{\frac{2}{3}}\frac{4\left(  t-\tau\right)  x^{2}%
}{1+2\left(  t-\tau\right)  ^{2}x^{2}}.
\]
Since $\ln\left(  1+z^{2}\right)  \leq\left\vert z\right\vert ,$ the following
inequality holds
\[
\left\vert v_{t}\left(  t,\tau,x\right)  \right\vert \leq\frac{2\sqrt{2}}%
{3}\left\vert \bar{a}\right\vert \left(  t-\tau\right)  ^{\frac{2}{3}%
}\left\vert x\right\vert +2\left\vert \bar{a}\right\vert \left(
t-\tau\right)  ^{-\frac{1}{3}}.
\]
Let us put $c_{0}\left(  t,\tau\right)  =\frac{2\sqrt{2}}{3}\left\vert \bar
{a}\right\vert \left(  t-\tau\right)  ^{\frac{2}{3}}$ and $d_{0}\left(
t,\tau\right)  =2\left\vert \bar{a}\right\vert \left(  t-\tau\right)
^{-\frac{1}{3}}.$ In an elementary way can be checked that $c_{0},d_{0}\in
L^{2}\left(  P_{\Delta},\mathbb{R}^{+}\right)  $ and
\[
\left\Vert c_{0}\right\Vert _{L^{2}\left(  P_{\Delta},\mathbb{R}^{+}\right)
}^{2}=\frac{4}{35}\bar{a}^{2}.
\]
Thus $v$ satisfies assumptions $\left(  A1\right)  -\left(  A3\right)  $.
Consequently, Theorem \ref{Tw4.2} implies that, for any $t\in\left[
0,1\right]  $ and $a\in AC_{0}^{2},$ the equation
\[
x\left(  t\right)  +\bar{a}\int\limits_{0}^{t}\left(  t-\tau\right)
^{\frac{2}{3}}\ln\left(  1+2\left(  t-\tau\right)  ^{2}x^{2}\left(
\tau\right)  \right)  d\tau=a\left(  t\right)  \text{ }%
\]
possesses a unique solution $x=x_{a}\in AC_{0}^{2}$ and the operator
$AC_{0}^{2}\backepsilon a\rightarrow x_{a}\in AC_{0}^{2}$ is continuously
Fr\'{e}chet differentiable.
\end{example}

\begin{example}
For any $t\in\left[  0,T\right]  $ with arbitrarily fixed $0<T<\infty,$ the
equation
\[
x\left(  t\right)  +\int\limits_{0}^{t}w\left(  t-\tau\right)  z\left(
x\left(  \tau\right)  \right)  d\tau=a\left(  t\right)  \text{ }%
\]
can be viewed as nonlinear control system with a feedback loop. In this case
the feedback is built up by the multiplication of a nonlinear term $z$ with no
memory and a linear time invariant part with memory $w(t-\tau)$. Subsequently,
the function $a$ is the convolution of the input function and the inverse of
some given transfer. Finally, the function $w$ is the convolution of the
inverse of the feedback transfer function and the inverse of the transfer
function, while $x$ can be interpreted as the output of the system. For
details, see, among others \cite{GriLonSta, MFa, MouMilMic}.

\noindent The nonlinear feedback term $z\in C^{1}(\mathbb{R},\mathbb{R})$ is
required to be such that there exist $A,B>0$ satisfying%
\[
\left\vert z\left(  x\right)  \right\vert \leq A\left\vert x\right\vert +B
\]
for all $x\in\mathbb{R}$. Moreover, we demand that $w\in C^{1}(\left[
0,T\right]  ,\mathbb{R})$ is such that $w\left(  0\right)  =0$ and
\[
\int\limits_{0}^{T}\int\limits_{0}^{t}\left\vert w^{\prime}\left(
t-\tau\right)  \right\vert ^{2}d\tau dt<\frac{1}{2A^{2}T^{2}}.
\]
Then the function $v(t,\tau,x)=w\left(  t-\tau\right)  z\left(  x\right)  $
satisfies assumptions $\left(  A1\right)  -\left(  A3\right)  $. Hence Theorem
\ref{Tw4.2} implies that, for any $t\in\left[  0,T\right]  $ and any $a\in
AC_{0}^{2},$ the equation
\[
x\left(  t\right)  +\int\limits_{0}^{t}w\left(  t-\tau\right)  z\left(
x\left(  \tau\right)  \right)  d\tau=a\left(  t\right)
\]
possesses a unique solution $x=x_{a}\in AC_{0}^{2}$ and the operator
$AC_{0}^{2}\backepsilon a\rightarrow x_{a}\in AC_{0}^{2}$ is continuously
Fr\'{e}chet differentiable.
\end{example}

\section{Summary}

Integral operators and equations involving them are most commonly considered
in the space of square-integrable functions. Under suitable conditions one
usually proves existence and uniqueness theorems. In this paper the integral
operator $V$ was defined on the smaller space of $AC_{0}^{2}$ to obtain
stronger continuity results. We have shown that assumptions (A1), (A2) and
either (A3) or (A4) imply some sufficient condition for the operator
$V:AC_{0}^{2}\rightarrow AC_{0}^{2}$ to be a diffeomorphism, cf. Theorem
\ref{Tw4.2}. It should be underlined that in the proof of Lemma \ref{lemat3.3}
and Lemma \ref{lemat3.4} we have used the compactness of the embedding of the
space $AC_{0}^{2}$ into the space $C,$ not true if we replace $AC_{0}^{2}$
with $L^{2}$. This compact embedding implies that any weakly convergent in
$AC_{0}^{2}$ sequence $\left\{  z_{k}\right\}  $ is uniformly convergent, i.e.
convergent in $C$ in the sup-norm. Apparently in the case of $L^{2}$ space
such an implication does not hold. Therefore, cannot be proved, at least with
the method applied herein, that the operator $V:L^{2}\rightarrow L^{2}$ is a
diffeomorphism, yet it is likely to be a homeomorphism in that case.

\end{document}